\documentclass{imanum}

\jno{drnxxx}
\received{}
\revised{}

\usepackage[T1]{fontenc}
\usepackage{times}
\usepackage{epsfig,graphicx,tabularx,color}
\usepackage[cp850]{inputenc}
\usepackage[english]{babel}
\usepackage{amsmath,amssymb,amsfonts,amsthm,mathrsfs,comment,mathdots,multirow,tikz}
\usepackage{rotating, longtable, relsize}
\usepackage[f]{esvect}
\usepackage{times}
\usepackage{fancyhdr}
\usepackage{lipsum}
\usepackage{soul}

\usetikzlibrary{calc,3d}

\theoremstyle{definition}

\def\XXint#1#2#3{{\setbox0=\hbox{$#1{#2#3}{\int}$}
     \vcenter{\hbox{$#2#3$}}\kern-.5\wd0}}

\def\pFq#1#2{{\,}_{#1}F_{#2}}

\def\ud{{\rm\,d}}

\def\C{\mathbb{C}}

\def\R{\mathbb{R}}
\def\Sph{\mathbb{S}}
\def\T{\mathbb{T}}

\def\LL{\mathcal{L}}

\def\OO{\mathcal{O}}

\newcommand{\bs}[1]{\boldsymbol{#1}}

\def\kb{\bs{k}}
\def\rb{\bs{r}}
\def\xb{\bs{x}}
\def\yb{\bs{y}}

\def\pr(#1){\left({#1}\right)}
\def\br[#1]{\left[{#1}\right]}

\def\abs#1{\left|{#1}\right|}

\def\pFq#1#2{{\,}_{#1}F_{#2}}

\def\ii{{\rm i}}

\newcommand{\expmo}{\operatorname{expm1}}
\newcommand{\logop}{\operatorname{log1p}}

\setulcolor{green}
\soulregister\cite7
\soulregister\ref7
\soulregister\pageref7

\DeclareRobustCommand{\rchi}{{\mathpalette\irchi\relax}}
\newcommand{\irchi}[2]{\raisebox{\depth}{$#1\chi$}} 

\begin{document}

\title{Fast and accurate algorithms for the computation of spherically symmetric nonlocal diffusion operators on lattices}
\shorttitle{Fourier spectra of nonlocal diffusion}

\author{%
{\sc Yu Li and Richard Mika\"el Slevinsky\thanks{Corresponding author. Email: Richard.Slevinsky@umanitoba.ca}}\\[2pt]
Department of Mathematics, University of Manitoba, Winnipeg, Canada}
\shortauthorlist{Y. Li and R. M. Slevinsky}

\maketitle

\begin{abstract}
{We present a unified treatment of the Fourier spectra of spherically symmetric nonlocal diffusion operators. We develop numerical and analytical results for the class of kernels with weak algebraic singularity as the distance between source and target tends to $0$. Rapid algorithms are derived for their Fourier spectra with the computation of each eigenvalue independent of all others. The algorithms are trivially parallelizable, capable of leveraging more powerful compute environments, and the accuracy of the eigenvalues is individually controllable. The algorithms include a Maclaurin series and a full divergent asymptotic series valid for any $d$ spatial dimensions. Using Drummond's sequence transformation, we prove linear complexity recurrence relations for degree-graded sequences of numerators and denominators in the rational approximations to the divergent asymptotic series. These relations are important to ensure that the algorithms are efficient, and also increase the numerical stability compared with the conventional algorithm with quadratic complexity.}
{Nonlocal diffusion; asymptotic series; Drummond's sequence transformation.}
\end{abstract}

\section{Introduction}

Nonlocal models are important for their fidelity and versatility in handling a broad range of applications in material science, thermodynamics, fluid dynamics, and image processing (\cite{Silling-48-175-00,Gilboa-Osher-7-1005-09,Estrada-et-al-78-1155-18,Du-Tian-1805-08261}). Mathematically, nonlocal diffusion is usually formulated through weakly singular integral operators, as in~\cite{Du-et-al-54-667-12}. A nonlocal vector calculus developed by~\cite{Du-et-al-23-493-13} has extended the principles of nonlocal diffusion to the vector operators of divergence, gradient, and curl and their adjoints.

Let $u$ be a multivariate function defined on the $d$-dimensional torus $\T^d := [0,2\pi)^d$, and by periodic extension to $\R^d$. We define nonlocal diffusion as
\begin{equation}
\LL_\delta u(\xb)= \int_{B_\delta(\xb)}\rho_\delta(\abs{\xb-\yb})[u(\yb)-u(\xb)] \ud\yb.
\end{equation}
where $B_\delta(\xb) = \{\yb \in\R^d : \abs{\xb-\yb}\le \delta\}$ is the closed ball of radius $0<\delta<\infty$ centred at $\xb$ and the kernel $\rho_\delta$ is a nonnegative spherically symmetric function of the euclidean distance compactly supported on $[0,\delta]$. Nonlocal diffusion has also been analyzed in different geometries, including the real line by~\cite{Zheng-et-al-39-A1951-17} and the unit $2$-sphere by~\cite{Slevinsky-Montanelli-Du-372-893-18}.

In $1$, $2$, and $3$ spatial dimensions,~\cite{Du-Yang-332-118-17} have developed algorithms for the numerical evaluation of nonlocal diffusion of Fourier series. For each dimension, this is done by showing that the Fourier basis are eigenfunctions of $\LL_\delta$ and by computing the spectra numerically. Three algorithms are proposed for the numerical evaluation of the Fourier spectra. Firstly, a Maclaurin series is obtained for each eigenvalue and for each spatial dimension. This Maclaurin series is extremely effective for eigenmodes with euclidean norm on the wavenumber multiplied by the nonlocal horizon $\delta$ is sufficiently small. Secondly, the integrals defining the eigenvalues are approximated by the (slowly converging) multivariate midpoint rule. Thirdly, the eigenvalues are demonstrated to be pointwise evaluations of the solution of a linear ordinary differential equation (ODE) with suitable initial values, that is executed using the Runge--Kutta method (RK4). Each algorithm has advantages and disadvantages that lead Du and Yang to propose a hybrid algorithm based on the Maclaurin series and the time-stepping of the ODE.

In this work, we present a unified treatment of the Fourier spectra of spherically symmetric nonlocal diffusion operators. Some of the results of Du and Yang are generalized to hold in $d$ dimensions. A natural reduction of dimensionality is present for all spherically symmetric kernels. The one dimensional integral representation is useful on its own because there are many reasonable choices for the kernel $\rho_\delta$ and thus the Fourier spectra for any kernel in any number of spatial dimensions may be considered from the same analytical position. Following Du and Yang, we continue to develop numerical and analytical results for the class of kernels with weak algebraic singularity as the distance between source and target tends to $0$. Rapid algorithms are derived for the computation of the Fourier spectra of $\LL_\delta$ with the computation of each eigenvalue independent of all others. The algorithms are trivially parallelizable, capable of leveraging more powerful compute environments, and the accuracy of the eigenvalues is individually controllable. The algorithms include a Maclaurin series and a full divergent asymptotic series valid for any $d$ spatial dimensions. The choice of such formul\ae~is based on the comparative study for semi-infinite integrals by~\cite{Slevinsky-Safouhi-60-315-12}, whereby it is recommended that divergent asymptotic series are resummed through the use of sequence transformations.

Certain classes of sequence transformations construct rational approximations to formal power series (\cite{Drummond-6-69-72,Levin-B3-371-73,Sidi-7-37-81,Weniger-10-189-89}). Rational approximants, such as Pad\'e approximants, are useful tools in the theory of analytic continuation and the summation of divergent series. Using Drummond's sequence transformation, we prove linear complexity recurrence relations for degree-graded sequences of numerators and denominators in the rational approximants. These relations are important to ensure that the algorithms are efficient, and also increase the numerical stability compared with the conventional algorithm with quadratic complexity.

Numerical evidence confirms that a very simple hybrid algorithm between the Maclaurin series and the asymptotic series performs favourably: microsecond numerical evaluation to approximately full double precision is delivered.

\section{Reduction of dimensionality}

As a complete orthogonal basis for $L^2(\T^d)$, Fourier series are indispensable in numerical analysis and approximation theory. Fourier series also greatly simplify the computation of nonlocal diffusion, since a straightforward calculation translating coordinates demonstrates that the Fourier modes are periodic eigenfunctions of $\LL_\delta$ on $\T^d$. That is,
\begin{equation}
\LL_\delta e^{\ii \kb\cdot\xb} = \lambda_\delta(\kb)e^{\ii\kb\cdot\xb},
\end{equation}
where
\begin{equation}
\lambda_\delta(\kb) = \int_{B_\delta(\bs{0})} \rho_\delta(\abs{\rb})\left(e^{\ii \kb\cdot\rb}-1\right)\ud\rb.
\end{equation}

\begin{theorem}
Let $k=\abs{\kb}$ and $r=\abs{\rb}$. The Fourier spectra of nonlocal diffusion may be reduced to the following one-dimensional integral
\begin{equation}
\lambda_\delta(\kb) = \int_0^\delta \rho_\delta(r)\left[(2\pi)^{\frac{d}{2}}\frac{J_{\frac{d-2}{2}}(kr)}{(kr)^{\frac{d-2}{2}}} - \frac{2\pi^{\frac{d}{2}}}{\Gamma(\frac{d}{2})}\right]r^{d-1}\ud r,
\end{equation}
where $J_\nu(\cdot)$ is the cylindrical Bessel function of the first kind~\cite[\S 10.2.2]{Olver-et-al-NIST-10} of order $\nu$ and $\Gamma(\cdot)$ is the gamma function~\cite[\S 5.2(i)]{Olver-et-al-NIST-10}.
\end{theorem}
\begin{proof}
The scalar product $\kb\cdot\rb = k r \cos\theta$ may be parameterized by the moduli of the vectors $\kb$ and $\rb$ and the angle $\theta\in[0,\pi]$ between them. As the integral is spherically symmetric, we rotate our frame of reference such that $\kb$ is oriented in the direction of the North hyperpole. Then, we may use hyperspherical coordinates with $r\in(0,\delta)$, $\theta = \theta_1$ and $\theta_i\in[0,\pi]$ for $i = 1,\ldots,d-2$, and $\theta_{d-1}\in[0,2\pi)$ to separate the multidimensional integrals defining $\lambda_\delta(\kb)$. Then
\begin{equation}
\lambda_\delta(\kb) = \int_0^{2\pi}\ud\theta_{d-1}\int_0^\pi\sin(\theta_{d-2})\ud\theta_{d-2}\cdots\int_0^\pi\sin^{d-2}(\theta)\int_0^\delta \rho_\delta(r)\left(e^{\ii kr\cos\theta}-1\right)r^{d-1}\ud r\ud\theta.
\end{equation}
All but the last angular integral may be represented by gamma functions~\cite[\S 
5.12.2]{Olver-et-al-NIST-10}
\begin{equation}
\int_0^\pi \sin^{d-1}\theta\ud\theta = \frac{\Gamma(\frac{1}{2})\Gamma(\frac{d}{2})}{\Gamma(\frac{d+1}{2})}.
\end{equation}
Then, by canceling common gamma functions appearing in the numerator and denominator,
\begin{align}
\lambda_\delta(\kb) & = 2\pi \frac{\Gamma(\frac{1}{2})\Gamma(\frac{2}{2})}{\Gamma(\frac{3}{2})} \frac{\Gamma(\frac{1}{2})\Gamma(\frac{3}{2})}{\Gamma(\frac{4}{2})} \cdots \frac{\Gamma(\frac{1}{2})\Gamma(\frac{d-2}{2})}{\Gamma(\frac{d-1}{2})}\int_0^\pi\sin^{d-2}(\theta)\int_0^\delta \rho_\delta(r)\left(e^{\ii kr\cos\theta}-1\right)r^{d-1}\ud r\ud\theta,\\
& = \frac{2\pi^{\frac{d-1}{2}}}{\Gamma(\frac{d-1}{2})}\int_0^\pi\sin^{d-2}(\theta)\int_0^\delta \rho_\delta(r)\left(e^{\ii kr\cos\theta}-1\right)r^{d-1}\ud r\ud\theta.
\end{align}
For the final angular integral, we use the convergent ultraspherical expansion of the plane wavefunction~\cite[\S 10.23.9]{Olver-et-al-NIST-10} and integrate term by term
\begin{align}
&\int_0^\pi \sin^{d-2}(\theta)\left(e^{\ii kr\cos\theta}-1\right)\ud\theta\nonumber\\
& = \int_0^\pi \sin^{d-2}(\theta)\left[2^{\frac{d-2}{2}}\Gamma(\tfrac{d-2}{2})\sum_{n=0}^\infty(\tfrac{d-2}{2}+n)\ii^n\frac{J_{\frac{d-2}{2}+n}(kr)}{(kr)^{\frac{d-2}{2}}}C_n^{(\frac{d-2}{2})}(\cos\theta) - 1\right]\ud\theta,\\
& = 2^{\frac{d-2}{2}}\Gamma(\tfrac{d-2}{2})\sum_{n=0}^\infty(\tfrac{d-2}{2}+n)\ii^n\frac{J_{\frac{d-2}{2}+n}(kr)}{(kr)^{\frac{d-2}{2}}}\int_0^\pi \sin^{d-2}(\theta)C_n^{(\frac{d-2}{2})}(\cos\theta)\ud\theta - \int_0^\pi \sin^{d-2}(\theta)\ud\theta.
\end{align}
By (weighted) orthogonality of the ultraspherical polynomials, every one of the integrals is $0$ except for the case $n=0$, completing the proof. For $d=1$, the ultraspherical expansion is unnecessary and for $d=2$, we may not use the ultraspherical expansion {\em per se}, but only by taking the limit as $d\to2$. In this case, the Chebyshev expansion of the plane wavefunction~\cite[\S 10.12.3]{Olver-et-al-NIST-10},
\begin{equation}
e^{\ii kr\cos\theta} = J_0(kr) + 2\sum_{n=1}^\infty \ii^n J_n(kr) T_n(\cos\theta),
\end{equation}
appears instead.
\end{proof}

An important class of kernels have a weak singularity as $\abs{\xb-\yb}\searrow0$. These kernels are also known to ensure strong convergence to local diffusion given by the Laplacian when the horizon $\delta\searrow0$, see~\cite{Du-Yang-332-118-17}, provided they are normalized by
\begin{equation}
\int_0^\delta \rho_\delta(r)r^{d+1}\ud r = \frac{2}{V_d} = \frac{2\Gamma(\frac{d}{2}+1)}{\pi^{\frac{d}{2}}},
\end{equation}
where $V_d$ is the volume of the $d$-dimensional unit hypersphere $\Sph^{d-1}\subset\R^d$. Motivated by the above, our choice of kernel is
\begin{equation}
\rho_\delta(r) = \frac{2\Gamma(\frac{d}{2}+1)(d+2-\alpha)}{\pi^{\frac{d}{2}}\delta^{d+2-\alpha}r^\alpha}\rchi_{[0,\delta]}(r),\quad{\rm for}\quad \alpha\in[0,d+2).
\end{equation}
Here $\rchi_{[0,\delta]}(\cdot)$ is the indicator function.

With this particular choice of a kernel, the Fourier spectra are
\begin{equation}\label{eq:FourierSpectra}
\lambda_\delta(\kb) = \frac{4\Gamma(\frac{d}{2}+1)(d+2-\alpha)}{\delta^{d+2-\alpha}} \int_0^\delta \left[\left(\frac{2}{kr}\right)^{\frac{d-2}{2}}J_{\frac{d-2}{2}}(kr) - \frac{1}{\Gamma(\frac{d}{2})}\right]r^{d-1-\alpha}\ud r.
\end{equation}

\subsection{Maclaurin series}

By using the cylindrical Bessel function's convergent Frobenius series~\cite[\S 10.2.2]{Olver-et-al-NIST-10},
\begin{equation}
\left(\frac{2}{kr}\right)^{\frac{d-2}{2}}J_{\frac{d-2}{2}}(kr) = \sum_{n=0}^\infty\frac{(-k^2r^2/4)^n}{n!\Gamma(n+\frac{d}{2})},
\end{equation}
the constant term on the right-hand side of Eq.~\eqref{eq:FourierSpectra} exactly cancels, resulting in
\begin{equation}
\lambda_\delta(\kb) = \frac{4\Gamma(\frac{d}{2}+1)(d+2-\alpha)}{\delta^{d+2-\alpha}} \int_0^\delta \sum_{n=1}^\infty\frac{(-k^2r^2/4)^n}{n!\Gamma(n+\frac{d}{2})}r^{d-1-\alpha}\ud r.
\end{equation}
Integrating term by term, we find the Maclaurin series
\begin{equation}\label{eq:MaclaurinSpectra}
\lambda_\delta(\kb) = \frac{4\Gamma(\tfrac{d}{2}+1)}{\delta^2}\sum_{n=1}^\infty \dfrac{(-k^2\delta^2/4)^n}{n!\Gamma(n+\frac{d}{2})}\frac{d+2-\alpha}{d+2n-\alpha}.
\end{equation}
At this point, we note that this series has been reported previously by~\cite{Du-Yang-332-118-17} for the dimensions $d=1,2,3$, and the difference in appears is because Eq.~\eqref{eq:MaclaurinSpectra} is valid for any number of dimensions.

As for algorithmic considerations, alternating series with terms with possibly large magnitude suffer from subtractive cancellation. In the case of Eq.~\eqref{eq:MaclaurinSpectra}, this corresponds to a large $k\delta$, which may occur for any fixed $\delta$ so long as the eigenvalues of sufficiently high Fourier modes are requested.

Before continuing, it will be fruitful to re-express Eq.~\eqref{eq:MaclaurinSpectra} as a generalized hypergeometric function. Let
\begin{equation}
(x)_n = \frac{\Gamma(x+n)}{\Gamma(x)},
\end{equation}
be the Pochhammer symbol for the rising factorial~\cite[\S 5.2(iii)]{Olver-et-al-NIST-10}. Generalized hypergeometric functions are a class of functions that are formally defined by their Maclaurin series
\begin{equation}
\pFq{p}{q}\left(\begin{array}{c}
a_1,\ldots,a_p\\
b_1,\ldots,b_q\end{array};
z\right) = \sum_{n=0}^\infty \dfrac{(a_1)_n\cdots(a_p)_n}{(b_1)_n\cdots(b_q)_n}\dfrac{z^n}{n!}.
\end{equation}
When one of the $a_i$ is a non-positive integer, then the series terminates after $a_i+1$ terms. Otherwise, the series has radius of convergence: $0$ if $p > q+1$; $1$ if $p = q+1$; and, $\infty$ if $p \le q$.
In the formalism of generalized hypergeometric functions, it is straightforward to confirm that
\begin{equation}\label{eq:FourierSpectraMAC}
\lambda_\delta(\kb) = -k^2\pFq{2}{3}\left(\begin{array}{c} 1, \tfrac{d+2-\alpha}{2}\\2, \tfrac{d+2}{2},\tfrac{d+4-\alpha}{2}\end{array}; -\frac{k^2\delta^2}{4}\right).
\end{equation}
Since $0<\alpha<d+2$, the representation as generalized hypergeometric functions is not terminating.

\subsection{Asymptotic series}

It is important to have a $d$-dimensional generalization of the Fourier spectra of nonlocal diffusion. However, the main computational issue is that as $k\to\infty$, the Maclaurin series is numerically ineffective. Starting with Eq.~\eqref{eq:FourierSpectra}, we develop an alternative formula with a full asymptotic expansion. It is tempting to find formul\ae~for each of the two integrals on the right-hand side of Eq.~\eqref{eq:FourierSpectra} independently; however, by ensuring integrability of each integrand, respectively, we retain an undue constraint on the strength of the algebraic singularity, $\alpha\in[0,d)$ rather than $\alpha\in[0,d+2)$. This forces us to use an alternative derivation valid for the interval $\alpha\in(\tfrac{d-1}{2},d+2)$, and piece together the result.

\subsubsection{A formula valid for $\alpha\in[0,d)$}

By using~\cite[\S 6.561 13.]{Gradshteyn-Ryzhik-07},
\begin{equation}\label{eq:PowerBesselIntegral}
\int_0^1 x^\mu J_\nu(ax)\ud x = \dfrac{2^\mu\Gamma(\frac{\nu+\mu+1}{2})}{a^{\mu+1}\Gamma(\frac{\nu-\mu+1}{2})} + a^{-\mu}\left[(\mu+\nu-1)J_\nu(a)S_{\mu-1,\nu-1}(a)-J_{\nu-1}(a)S_{\mu,\nu}(a)\right],
\end{equation}
where $S_{\mu,\nu}(\cdot)$ is a Lommel function~\cite[\S11.9]{Olver-et-al-NIST-10} and which is valid for $a>0$ and $\Re(\mu+\nu)>-1$, we may express Eq.~\eqref{eq:FourierSpectra} as
\begin{subequations}\label{eq:FourierSpectraASY}
\begin{align}
& \frac{\delta^2\lambda_\delta(\kb)}{2\Gamma(\frac{d}{2}+1)(d+2-\alpha)} = \left(\frac{k\delta}{2}\right)^{\alpha-d}\frac{\Gamma(\frac{d-\alpha}{2})}{\Gamma(\frac{\alpha}{2})} - \frac{2}{(d-\alpha)\Gamma(\frac{d}{2})}\label{eq:FourierSpectraASYa}\\
& \quad\quad + 2^{\frac{d}{2}}(k\delta)^{\alpha+1-d}\left[(d-2-\alpha)J_{\frac{d-2}{2}}(k\delta)S_{\frac{d-2-2\alpha}{2},\frac{d-4}{2}}(k\delta) - J_{\frac{d-4}{2}}(k\delta)S_{\frac{d-2\alpha}{2},\frac{d-2}{2}}(k\delta)\right].\label{eq:FourierSpectraASYb}
\end{align}
\end{subequations}
Eq.~\eqref{eq:FourierSpectraASY} separates the Fourier spectra into the difference of the ratio of gamma functions in Eq.~\eqref{eq:FourierSpectraASYa} and a scaled sum of Bessel and Lommel functions in Eq.~\eqref{eq:FourierSpectraASYb}. In particular, the Bessel and Lommel functions are analytic functions on any subset of the positive real axis, independent of the order and parameters. Therefore, Eq.~\eqref{eq:FourierSpectraASYa}, with its difference of singularities as $\alpha\to d^-$, is the only part of the formula that challenges a pedestrian analytic continuation.

We close this section by noting that the Lommel functions appearing in Eq.~\eqref{eq:FourierSpectraASYb} have an asymptotic expansion~\cite[\S 8.576]{Gradshteyn-Ryzhik-07} given by
\begin{align}
S_{\mu,\nu}(z) & \sim z^{\mu-1}\sum_{n=0}^\infty (-1)^n \left(\frac{1-\mu+\nu}{2}\right)_n\left(\frac{1-\mu-\nu}{2}\right)_n \left(\frac{2}{z}\right)^{2n},\quad{\rm as}\quad z\to\infty,\label{eq:LommelS2}\\
& = z^{\mu-1}\pFq{3}{0}\left(\begin{array}{c} 1, \dfrac{1-\mu+\nu}{2}, \dfrac{1-\mu-\nu}{2}\end{array};-\frac{4}{z^2}\right).
\end{align}

\subsubsection{A formula valid for $\alpha\in(\tfrac{d-1}{2},d+2)$}

Eq.~\eqref{eq:FourierSpectra} may be written as
\begin{align}
\frac{\delta^{d+2-\alpha}\lambda_\delta(\kb)}{4\Gamma(\frac{d}{2}+1)(d+2-\alpha)} & = \int_0^\infty \left[\left(\frac{2}{kr}\right)^{\frac{d-2}{2}}J_{\frac{d-2}{2}}(kr) - \frac{\rchi_{[0,\delta]}(r)}{\Gamma(\frac{d}{2})}\right]r^{d-1-\alpha}\ud r\nonumber\\
& \quad - \int_\delta^\infty \left(\frac{2}{kr}\right)^{\frac{d-2}{2}}J_{\frac{d-2}{2}}(kr)r^{d-1-\alpha}\ud r.
\end{align}
Next, we introduce an $\epsilon$-limit in the first integral, weakening the singularity at the origin and thereby allowing integrals to be separated,
\begin{align}
\frac{\delta^{d+2-\alpha}\lambda_\delta(\kb)}{4\Gamma(\frac{d}{2}+1)(d+2-\alpha)} & = \lim_{\epsilon\to0^+}\int_0^\infty \left[\left(\frac{2}{kr}\right)^{\frac{d-2}{2}}J_{\frac{d-2}{2}}(kr) - \frac{\rchi_{[0,\delta]}(r)}{\Gamma(\frac{d}{2})}\right](r+\epsilon)^{d-1-\alpha}\ud r\nonumber\\
& \quad - \int_\delta^\infty \left(\frac{2}{kr}\right)^{\frac{d-2}{2}}J_{\frac{d-2}{2}}(kr)r^{d-1-\alpha}\ud r,\\
& = \lim_{\epsilon\to0^+}\left[\int_0^\infty \left(\frac{2}{kr}\right)^{\frac{d-2}{2}}J_{\frac{d-2}{2}}(kr)(r+\epsilon)^{d-1-\alpha}\ud r - \int_0^\delta\frac{(r+\epsilon)^{d-1-\alpha}}{\Gamma(\frac{d}{2})}\ud r\right]\nonumber\\
& \quad - \int_\delta^\infty \left(\frac{2}{kr}\right)^{\frac{d-2}{2}}J_{\frac{d-2}{2}}(kr)r^{d-1-\alpha}\ud r.
\end{align}

We utilize~\cite[\S 6.563]{Gradshteyn-Ryzhik-07},
\begin{align}
\int_0^\infty x^{\varrho-1}J_\nu(ax)\frac{\ud x}{(x+\epsilon)^{1+\mu}}=&\frac{\pi \epsilon^{\varrho-\mu-1}}{\sin[(\varrho+\nu-\mu)\pi]\Gamma(\mu+1)}\nonumber\\
&\times\left[\sum_{n=0}^{\infty}\frac{(-1)^n (\frac{1}{2}a\epsilon)^{\nu+2n}\Gamma(\varrho+\nu+2n)}{n!\Gamma(\nu+n+1)\Gamma(\varrho+\nu-\mu+2n)}\right. \\
&- \left.\sum_{n=0}^{\infty} \frac{(\frac{1}{2}a\epsilon)^{\mu+1-\varrho+n}\Gamma(\mu+n+1)\sin [ \frac{1}{2}(\varrho+\nu-\mu-n)\pi ]}{n!\Gamma [ \frac{1}{2}(\mu+\nu-\varrho+n+3) ] \Gamma [ \frac{1}{2}(\mu-\nu-\varrho+n+3) ] }\right]\nonumber
\end{align}
where $a>0$, $|\arg \epsilon| < \pi$, $\Re(\varrho+\nu)>0$, and $\Re(\varrho-\mu)<\frac{5}{2}$, with $\varrho = 2-\frac{d}{2}$, $\nu=\frac{d-2}{2}$, $\mu = \alpha-d$.

Since
\begin{equation}
\int_0^\delta\frac{(r+\epsilon)^{d-1-\alpha}}{\Gamma(\frac{d}{2})}\ud r = \frac{(\delta+\epsilon)^{d-\alpha}}{(d-\alpha)\Gamma(\frac{d}{2})} - \frac{\epsilon^{d-\alpha}}{(d-\alpha)\Gamma(\frac{d}{2})},
\end{equation}
we may take the limit
\begin{align}
\lim_{\epsilon\to0^+}\int_0^\infty &\left[\left(\frac{2}{kr}\right)^{\frac{d-2}{2}}J_{\frac{d-2}{2}}(kr) - \frac{\rchi_{[0,\delta]}(r)}{\Gamma(\frac{d}{2})}\right](r+\epsilon)^{d-1-\alpha}\ud r\nonumber\\
& = \lim_{\epsilon\to0^+} \left\{\left(\frac{2}{k}\right)^{\frac{d-2}{2}}\frac{\pi \epsilon^{\frac{d+2}{2}-\alpha}}{\sin[(d-\alpha+1)\pi]\Gamma(\alpha-d+1)}\right.\nonumber\\
&\times\left[\sum_{n=0}^{\infty}\frac{(-1)^n (\frac{1}{2}k\epsilon)^{\frac{d-2}{2}+2n}\Gamma(2n+1)}{n!\Gamma(n+\frac{d}{2})\Gamma(2n+d-\alpha+1)}\right. \\
&- \left.\sum_{n=0}^{\infty} \frac{(\frac{1}{2}k\epsilon)^{\alpha-\frac{d+2}{2}+n}\Gamma(n+\alpha-d+1)\sin [ \frac{1}{2}(d-\alpha+1-n)\pi ]}{n!\Gamma [ \frac{1}{2}(n+\alpha) ] \Gamma [ \frac{1}{2}(n+\alpha-d+2) ] }\right]\nonumber\\
& \left.- \frac{(\delta+\epsilon)^{d-\alpha}}{(d-\alpha)\Gamma(\frac{d}{2})} + \frac{\epsilon^{d-\alpha}}{(d-\alpha)\Gamma(\frac{d}{2})}\right\}.\nonumber
\end{align}
In the first infinite series, the power of $\epsilon$ of the $n^{\rm th}$ term is $d-\alpha+2n$. Since $\alpha<d+2$, the zeroth term is singular, and every other power of $\epsilon$ is positive. In the second infinite series, the power of $\epsilon$ of the $n^{\rm th}$ term is $n$, ensuring that only the $n=0$ terms remain in the limit. Therefore
\begin{align}
\lim_{\epsilon\to0^+}\int_0^\infty &\left[\left(\frac{2}{kr}\right)^{\frac{d-2}{2}}J_{\frac{d-2}{2}}(kr) - \frac{\rchi_{[0,\delta]}(r)}{\Gamma(\frac{d}{2})}\right](r+\epsilon)^{d-1-\alpha}\ud r\nonumber\\
& = \lim_{\epsilon\to0^+} \left\{\left(\frac{2}{k}\right)^{\frac{d-2}{2}}\frac{\pi \epsilon^{\frac{d+2}{2}-\alpha}}{\sin[(d-\alpha+1)\pi]\Gamma(\alpha-d+1)}\right.\nonumber\\
&\times\left[\frac{(\frac{1}{2}k\epsilon)^{\frac{d-2}{2}}}{\Gamma(\frac{d}{2})\Gamma(d-\alpha+1)} - \frac{(\frac{1}{2}k\epsilon)^{\alpha-\frac{d+2}{2}}\Gamma(\alpha-d+1)\sin [ \frac{1}{2}(d-\alpha+1)\pi ]}{\Gamma(\frac{\alpha}{2}) \Gamma [ \frac{1}{2}(\alpha-d+2) ] }\right]\\
& \left.- \frac{(\delta+\epsilon)^{d-\alpha}}{(d-\alpha)\Gamma(\frac{d}{2})} + \frac{\epsilon^{d-\alpha}}{(d-\alpha)\Gamma(\frac{d}{2})}\right\}.\nonumber
\end{align}
Simplifying further,
\begin{align}
\lim_{\epsilon\to0^+}\int_0^\infty &\left[\left(\frac{2}{kr}\right)^{\frac{d-2}{2}}J_{\frac{d-2}{2}}(kr) - \frac{\rchi_{[0,\delta]}(r)}{\Gamma(\frac{d}{2})}\right](r+\epsilon)^{d-1-\alpha}\ud r\nonumber\\
& = \lim_{\epsilon\to0^+} \left\{\frac{\pi\epsilon^{d-\alpha}}{\sin[(d-\alpha+1)\pi]\Gamma(\alpha-d+1)\Gamma(\frac{d}{2})\Gamma(d-\alpha+1)} + \frac{\epsilon^{d-\alpha}}{(d-\alpha)\Gamma(\frac{d}{2})}\right.\\
& \left. - \left(\frac{k}{2}\right)^{\alpha-d} \frac{\pi \sin[\frac{1}{2}(d-\alpha+1)\pi]}{\sin[(d-\alpha+1)\pi]\Gamma(\frac{\alpha}{2})\Gamma[\frac{1}{2}(\alpha-d+2)]} - \frac{(\delta+\epsilon)^{d-\alpha}}{(d-\alpha)\Gamma(\frac{d}{2})}\right\}.\nonumber
\end{align}
By using the reflection formul\ae~\cite[\S 8.334~2.~\&~3.]{Gradshteyn-Ryzhik-07} for the gamma function,
\begin{align}
\lim_{\epsilon\to0^+}\int_0^\infty &\left[\left(\frac{2}{kr}\right)^{\frac{d-2}{2}}J_{\frac{d-2}{2}}(kr) - \frac{\rchi_{[0,\delta]}(r)}{\Gamma(\frac{d}{2})}\right](r+\epsilon)^{d-1-\alpha}\ud r\nonumber\\
& = \lim_{\epsilon\to0^+} \left\{\frac{\epsilon^{d-\alpha}}{(\alpha-d)\Gamma(\frac{d}{2})} + \frac{\epsilon^{d-\alpha}}{(d-\alpha)\Gamma(\frac{d}{2})}\right.\\
& \left. + \left(\frac{k}{2}\right)^{\alpha-d} \frac{ \Gamma(\frac{d-\alpha}{2})}{2\Gamma(\frac{\alpha}{2})} - \frac{(\delta+\epsilon)^{d-\alpha}}{(d-\alpha)\Gamma(\frac{d}{2})}\right\}\nonumber\\
& = \left(\frac{k}{2}\right)^{\alpha-d} \frac{ \Gamma(\frac{d-\alpha}{2})}{2\Gamma(\frac{\alpha}{2})} - \frac{\delta^{d-\alpha}}{(d-\alpha)\Gamma(\frac{d}{2})}.
\end{align}
Multiplying by $2\delta^{\alpha-d}$, we arrive at Eq.~\eqref{eq:FourierSpectraASYa}.

Next, we use the formula
\begin{equation}
\int_1^\infty x^\mu J_\nu(ax)\ud x = - a^{-\mu}\left[(\mu+\nu-1)J_\nu(a)S_{\mu-1,\nu-1}(a)-J_{\nu-1}(a)S_{\mu,\nu}(a)\right],
\end{equation}
valid for $a>0$ and $\Re(\mu)<\frac{1}{2}$, to arrive at Eq.~\eqref{eq:FourierSpectraASYb}. Compare~\cite[\S 6.561~13.~\&~14.]{Gradshteyn-Ryzhik-07} and see the discussion in~\cite[\S 10.74]{Watson-66} that relates Lommel functions to the indefinite integration of Bessel functions with an algebraic power. The restriction on $\mu$ is equivalent to $\alpha>\frac{d-1}{2}$.

\subsubsection{The removable singularity}
Notice that
\begin{equation}
\left(\frac{k\delta}{2}\right)^{\alpha-d}\frac{\Gamma(\frac{d-\alpha}{2})}{\Gamma(\frac{\alpha}{2})} - \frac{2}{(d-\alpha)\Gamma(\frac{d}{2})} = \frac{2}{(d-\alpha)\Gamma(\frac{\alpha}{2})}\left[\left(\frac{k\delta}{2}\right)^{\alpha-d}\Gamma(\tfrac{d-\alpha}{2}+1) - \frac{\Gamma(\frac{\alpha}{2})}{\Gamma(\frac{d}{2})}\right].
\end{equation}
This reorganization delays the onset of the pole field generated by the gamma function in the numerator by precisely two units in $\alpha$. By l'H\^opital's rule,
\begin{equation}
\lim_{\alpha\to d}\frac{2}{(d-\alpha)\Gamma(\frac{\alpha}{2})}\left[\left(\frac{k\delta}{2}\right)^{\alpha-d}\Gamma(\tfrac{d-\alpha}{2}+1) - \frac{\Gamma(\frac{\alpha}{2})}{\Gamma(\frac{d}{2})}\right] = \frac{1}{\Gamma(\frac{d}{2})}\left[2\log\left(\frac{2}{k\delta}\right) + \Gamma'(1) + \frac{\Gamma'(\frac{d}{2})}{\Gamma(\frac{d}{2})}\right].
\end{equation}
With a finite limit, we conclude that both sides of Eq.~\eqref{eq:FourierSpectraASY} are analytic functions of $\alpha\in[0,d+2)$.

\begin{remark}
While the lower bound on the interval of $\alpha$ is proposed to ensure that the algebraic kernel is indeed singular as $r\to0^+$, there is no issue of convergence in extending Eq.~\eqref{eq:FourierSpectraASY} to the interval $\alpha\in(-\infty,d)$. Since Eq.~\eqref{eq:FourierSpectraASY}, as a function of $\alpha$, is comprised of exponentials, gamma functions, Bessel functions, and Lommel functions, all of which satisfy linear homogeneous recurrence relations, the existence of a recurrence relation provides an alternative method to analytically continue Eq.~\eqref{eq:FourierSpectraASY} to the necessary parameter range $\alpha\in[0,d+2)\subset (-\infty,d+2)$. Since $\alpha\in(-\infty,d)$ provides infinitely-many initial conditions to this recurrence relation, there can be no other solution than Eq.~\eqref{eq:FourierSpectraASY}.
\end{remark}

\section{Numerical evaluation of $\lambda_\delta(\kb)$}

\subsection{Numerical evaluation of Eq.~\eqref{eq:FourierSpectraMAC}}

The Maclaurin series in Eq.~\eqref{eq:FourierSpectraMAC} is evaluated numerically as any generalized hypergeometric function. Since $k\delta > 0$, the series is alternating and the magnitude of the first omitted term is a bound on the absolute error. The partial sums of the series itself are tallied as an approximation to the sum and permit the implementation of a relative criterion for convergence. We stop the summation when the addition of the next term in the series is less than $10\epsilon_{\rm mach}\approx 2.22\times10^{-15}$ multiplied by the magnitude of partial sum of the series.

\subsection{Numerical evaluation of Eq.~\eqref{eq:FourierSpectraASYa}}

In this section, we discuss the numerical evaluation of the function
\begin{equation}
f(x,y,z) = \left[\dfrac{y^{2x}\Gamma(x+1)\Gamma(z)}{\Gamma(z-x)}-1\right]\Bigg/x,
\end{equation}
as it pertains to Eq.~\eqref{eq:FourierSpectraASYa} after setting $x = \frac{d-\alpha}{2}$, $y = \frac{k\delta}{2}>0$, and $z=\frac{d}{2}>0$. As $x\to0$, there is the possibility for subtractive cancellation. We also note the limiting value
\begin{equation}
\lim_{x\to0}f(x,y,z) = 2\log(y) + \frac{\Gamma'(1)}{\Gamma(1)} + \frac{\Gamma'(z)}{\Gamma(z)}.
\end{equation}
We implement the logarithmic derivative of the gamma function via the digamma function~\cite[\S 5.2(i)]{Olver-et-al-NIST-10}, $\psi(z)$.
To deal with the numerical instability as $x\to0$, we adapt a method due to~\cite{Michel-Stoitsov-178-535-08} that was part of a polyalgorithm for the numerical evaluation of Gauss' hypergeometric function. In the first instance, we rely on the fact that many programming languages have a special implementation of $\expmo(x) = \exp(x)-1$ and $\logop(x) = \log(1+x)$. Thus
\begin{equation}
f(x,y,z) = \expmo\left\{\log\left[\dfrac{y^{2x}\Gamma(x+1)\Gamma(z)}{\Gamma(z-x)}\right]\right\}\Bigg/x.
\end{equation}
It is important to note that for $d=1,2,\ldots$, and for $\alpha\in[0,d+2)$, the permissible values of $z-x\in[0,\tfrac{d+2}{2})$ and $x+1\in(0,\tfrac{d+2}{2}]$, which means we do not need to worry about numerical evaluation near any other singularity other than as $x\to0$.

Since
\begin{equation}
\log\frac{y^{2x}\Gamma(x+1)\Gamma(z)}{\Gamma(z-x)} = 2x\log(y) + \log\frac{\Gamma(1+x)}{\Gamma(1)} - \log\frac{\Gamma(z-x)}{\Gamma(z)},
\end{equation}
we use the Lanczos approximation~\cite{Lanczos-1-86-64},
\begin{equation}\label{eq:Lanczos}
\Gamma(z+1) \approx \sqrt{2\pi}\left(z+\gamma+\frac{1}{2}\right)^{z+\frac{1}{2}}e^{-(z+\gamma+\frac{1}{2})}\left(c_0+\sum_{i=1}^n\frac{c_i}{z+i}\right),
\end{equation}
for the logarithm of the ratio of gamma functions
\begin{align}
\log\frac{\Gamma(z+1+\epsilon)}{\Gamma(z+1)} & \approx \left(z+\frac{1}{2}\right)\logop\left(\frac{\epsilon}{z+\gamma+\frac{1}{2}}\right)\nonumber\\
& \quad + \epsilon\log\left(z+\gamma+\frac{1}{2}+\epsilon\right) - \epsilon + \logop\left[-\epsilon\frac{\sum_{i=1}^n\frac{c_i}{(z+i)(z+i+\epsilon)}}{c_0+\sum_{i=1}^n\frac{c_i}{z+i}}\right].\label{eq:Lanczosratio}
\end{align}

In Eqs.~\eqref{eq:Lanczos} and~\eqref{eq:Lanczosratio}, we use the coefficients $c_0,\ldots,c_n$, and the constant $\gamma$ that are reported in the accompanying software {\tt hyp\_2F1} provided by~\cite{Michel-Stoitsov-178-535-08}. Eq.~\eqref{eq:Lanczos} is only valid for $\Re (z+1) > 0$, but Euler's reflection formula may be used to extend its applicability. Another reflection formula is developed for Eq.~\eqref{eq:Lanczosratio} by~\cite{Michel-Stoitsov-178-535-08}.

\subsection{Numerical evaluation of Eq.~\eqref{eq:FourierSpectraASYb}}

Let $\{a_k\}_{k=0}^\infty$ be an infinite sequence, let $s_n = \sum_{k=0}^n a_k$ denote the $n^{\rm th}$ partial sum of the sequence, and $s$ the limit of the partial sums as $n\to\infty$ or the antilimit in case the series is divergent. Sequence transformations are a family of methods that are designed to accelerate the convergence a sequence (of partial sums) to its limit, or potentially converge to its antilimit. We refer the interested reader to the references (\cite{Drummond-6-69-72,Levin-B3-371-73,Sidi-7-37-81,Weniger-10-189-89}) and the references therein for their rich history. Many sequence transformations for the summation of series begin with the {\em ansatz} that the difference between (anti)limit and the $n^{\rm th}$ partial sum has a prescribed form; the remainder is given by a particular controlling factor $\omega_n$ and a constructive correction term $z_n$,
\begin{equation}
s - s_n = \omega_n z_n.
\end{equation}
The controlling factor is usually motivated by an analysis of the type of sequence under consideration. For example, a large class of sequence transformations developed by~\cite{Levin-Sidi-9-175-81} depends on the sequence itself satisfying an $m^{\rm th}$ order linear homogeneous difference equation with variable coefficients that have Poincar\'e-type asymptotic expansions as $n\to\infty$. In the case of generalized hypergeometric series, the sequence is given by
\begin{equation}
a_k = \frac{(\alpha_1)_k\cdots(\alpha_p)_k}{(\beta_1)_k\cdots(\beta_q)_k}\frac{z^k}{k!},
\end{equation}
and since
\begin{equation}
a_{k+1} = \frac{(\alpha_1+k)\cdots(\alpha_p+k)}{(\beta_1+k)\cdots(\beta_q+k)}\frac{z}{k+1}a_k,\qquad a_0 = 1,
\end{equation}
it is easily verified that the sequence satisfies a first-order linear homogeneous difference equation. In such instances, it is legitimate for the controlling factor to depend on the sequence itself, $\omega_n = \Delta s_n = a_{n+1}$. We refer the interested reader for the systematic construction of such methods in~\cite{Levin-Sidi-9-175-81}.

With the controlling factor in hand, several different choices of the correction term have been considered (\cite{Drummond-6-69-72,Levin-B3-371-73,Sidi-7-37-81,Weniger-10-189-89}). Drummond's sequence transformation is designed to be exact for polynomial corrections
\begin{equation}
T_n^{(k)} - s_n = \Delta s_n p_{k-1}(n),
\end{equation}
where $\deg(p_{k-1}) = k-1$ and $p_{-1}(n) \equiv 0$. The approximant $T_n^{(k)}$ may be obtained by setting up a $k\times k$ linear system that utilizes the partial sums $s_n,\ldots,s_{n+k}$. Generically solving a linear system, however, runs in $\OO(k^3)$ operations and is considered rather prohibitive. Due to the particular structure of the sequence transformation, an $\OO(k^2)$ algorithm is readily obtained by dividing by the controlling factor, and by annihilating the polynomial correction with a $k^{\rm th}$ order (forward) finite difference,
\begin{equation}\label{eq:Drummondk^2}
T_n^{(k)} = \dfrac{N_n^{(k)}}{D_n^{(k)}} = \dfrac{\Delta^k\left(\dfrac{s_n}{\Delta s_n}\right)}{\Delta^k\left(\dfrac{1}{\Delta s_n}\right)} = \dfrac{\displaystyle \sum_{j=0}^k\binom{k}{j}(-1)^{k-j} \dfrac{s_{n+j}}{a_{n+j+1}}}{\displaystyle \sum_{j=0}^k\binom{k}{j}(-1)^{k-j} \dfrac{1}{a_{n+j+1}}}.
\end{equation}

In the case of Euler's factorially divergent asymptotic series for the exponential integral, Drummond's sequence transformation is known to produce the type $[n+k/k]$ Pad\'e approximant, see for example~\cite{Borghi-Weniger-94-149-15}. By identifying Euler's series with a Stieltjes moment problem whose moments satisfy Carleman's condition, the pointwise root-exponential convergence of diagonals in the table\footnote{By diagonals, we mean fixed $n$ and increasing $k$.} of Drummond's sequence transformation has been rigorously established. More generally, Drummond's sequence transformation creates a class of Pad\'e-{\em type} approximants, see~\cite{Brezinski-80}, rational approximants with a fixed (sub)set of poles, and while less is known about the convergence of Pad\'e-type approximants to analytic continuations of divergent asymptotic series, our numerical experiments suggest that Drummond's sequence transformation converges pointwise to Lommel functions of positive argument as well.

Recurrence relations have been developed by~\cite{Weniger-10-189-89} to help transform the sequence of partial sums, $\{T_{n+j}^{(0)}\}_{j=0}^k$, to the fully transformed sequences $\{T_n^{(j)}\}_{j=0}^k$. However, these only avoid the construction of binomial coefficients and do not reduce the asymptotic complexity of the procedure.

Consider the summation of
\begin{equation}\label{eq:3F0special}
a_k = (\alpha)_k(\beta)_k/(-z)^k,
\end{equation}
which appears in the computation of the Lommel functions $S_{\mu,\nu}(z)$ via their asymptotic expansion in Eq.~\eqref{eq:LommelS2}. We apply Drummond's sequence transformation and obtain the following four-term recurrence relation. As is borne out in our numerical experiments, a linear complexity recurrence relation has a significant impact on the competitiveness of Drummond's sequence transformation via-\`a-vis direct summation of the (divergent) asymptotic series, and also increases the numerical stability of the transformations compared with the $\OO(k^2)$ procedure.

\begin{theorem}\label{theorem:DrummondRec}
When applied to the sequence in Eq.~\eqref{eq:3F0special}, both numerators and denominators of Drummond's sequence transformation satisfy the recurrence relation
\begin{align}
& (\alpha+n+k+1)(\beta+n+k+1)Q_n^{(k+1)}\nonumber\\
& \quad + \left[z+k(\alpha+\beta+2n+2k+1)+(\alpha+n+k+1)(\alpha+n+k+1)\right]Q_n^{(k)}\label{eq:Drummondk}\\
& \quad + k(\alpha+\beta+2n+3k)Q_n^{(k-1)}+k(k-1)Q_n^{(k-2)} = 0,\nonumber
\end{align}
though with the different initial conditions
\begin{equation}
D_n^{(0)} = \frac{1}{a_{n+1}},\quad{\rm and}\quad D_n^{(-1)} = D_n^{(-2)} = 0,
\end{equation}
and
\begin{equation}
N_n^{(1)} = s_nD_n^{(1)} + \frac{a_{n+1}}{a_{n+2}},\quad{\rm and}\quad N_n^{(0)} = s_nD_n^{(0)},\quad{\rm and}\quad N_n^{(-1)} = 0.
\end{equation}
\end{theorem}
\begin{proof}
Recall the finite product rule
\begin{equation}
\Delta(f_ng_n) = f_{n+1}\Delta g_n + \Delta f_n g_n = f_n\Delta g_n + \Delta f_n g_{n+1}.
\end{equation}
Assume that the recurrence relation holds for some $k$. The order $k$ of either $N_n^{(k)}$ or $D_n^{(k)}$ is incremented by applying another finite difference such that $Q_n^{(k+1)} = \Delta Q_n^{(k)}$. Using the first form of the finite product rule,
\begin{align}
& \Delta\left[(\alpha+n+k+1) (\beta+n+k+1)Q_n^{(k+1)}\right]\nonumber\\
& \quad + \Delta\left\{\left[z+k(\alpha+\beta+2n+2k+1)+(\alpha+n+k+1)(\beta+n+k+1)\right]Q_n^{(k)}\right\}\\
& \quad + \Delta\left[k(\alpha+\beta+2n+3k)Q_n^{(k-1)}\right]+\Delta[k(k-1)Q_n^{(k-2)}],\nonumber\\~\nonumber\\
& = (\alpha+n+k+2) (\beta+n+k+2)Q_n^{(k+2)}\nonumber\\
& \quad + \left[(\alpha+n+k+2) (\beta+n+k+2)-(\alpha+n+k+1) (\beta+n+k+1)\right]Q_n^{(k+1)}\nonumber\\
& \quad + \left[z+k(\alpha+\beta+2n+2k+3)+(\alpha+n+k+2)(\alpha+n+k+2)\right]Q_n^{(k+1)}\\
& \quad + \left[2k+(\alpha+n+k+2)(\alpha+n+k+2)-(\alpha+n+k+1)(\beta+n+k+1)\right]Q_n^{(k)}\nonumber\\
& \quad + k(\alpha+\beta+2n+3k+2)Q_n^{(k)} + 2kQ_n^{(k-1)}+k(k-1)Q_n^{(k-1)},\nonumber\\~\nonumber\\
& = (\alpha+n+k+2)(\beta+n+k+2)Q_n^{(k+2)}\nonumber\\
& \quad + (\alpha+\beta+2n+2k+3)Q_n^{(k+1)}\nonumber\\
& \quad + \left[z+k(\alpha+\beta+2n+2k+3)+(\alpha+n+k+2)(\beta+n+k+2)\right]Q_n^{(k+1)}\\
& \quad + (2k+\alpha+\beta+2n+2k+3)Q_n^{(k)}\nonumber\\
& \quad + k(\alpha+\beta+2n+3k+2)Q_n^{(k)} + k(k+1)Q_n^{(k-1)},\nonumber\\~\nonumber\\
& = (\alpha+n+k+2)(\beta+n+k+2)Q_n^{(k+2)}\nonumber\\
& \quad + \left[z+(k+1)(\alpha+\beta+2n+2k+3)+(\alpha+n+k+2)(\beta+n+k+2)\right]Q_n^{(k+1)}\\
& \quad + (k+1)(\alpha+\beta+2n+3k+3)Q_n^{(k)} + k(k+1)Q_n^{(k-1)} = 0.\nonumber
\end{align}
For the particular sequence at hand,
\begin{equation}
\frac{a_{n+1}}{a_{n+2}} = -\frac{z}{(\alpha+n+1)(\beta+n+1)}.
\end{equation}
Therefore, we may begin the denominator recurrence via
\begin{equation}
D_n^{(0)} = \frac{1}{a_{n+1}} = \frac{(-z)^{n+1}}{(\alpha)_{n+1}(\beta)_{n+1}},
\end{equation}
and
\begin{equation}
D_n^{(1)} = \frac{1}{a_{n+2}}-\frac{1}{a_{n+1}} = \left(\frac{a_{n+1}}{a_{n+2}}-1\right)\frac{1}{a_{n+1}} = -\left[\frac{z}{(\alpha+n+1)(\beta+n+1)}+1\right]D_n^{(0)},
\end{equation}
or
\begin{equation}
(\alpha+n+1)(\beta+n+1)D_n^{(1)} + \left[z+(\alpha+n+1)(\beta+n+1)\right]D_n^{(0)} = 0.
\end{equation}
But this is simply the recurrence relation with $k=0$.

Similarly, the numerator recurrence is begun with
\begin{equation}
N_n^{(0)} = \frac{s_n}{a_{n+1}} = s_n D_n^{(0)}.
\end{equation}
Using the second form of the finite product rule,
\begin{align}
N_n^{(1)} & = s_n \Delta D_n^{(0)} + \Delta s_n D_{n+1}^{(0)} = s_n D_n^{(1)} + \frac{\Delta s_n}{\Delta s_{n+1}}\\
& = -\left[\frac{z}{(\alpha+n+1)(\beta+n+1)}+1\right] s_n D_n^{(0)} + \frac{a_{n+1}}{a_{n+2}}\\
& = -\left[\frac{z}{(\alpha+n+1)(\beta+n+1)}+1\right] N_n^{(0)} - \frac{z}{(\alpha+n+1)(\beta+n+1)},
\end{align}
or
\begin{equation}
(\alpha+n+1)(\beta+n+1)N_n^{(1)} + \left[z+(\alpha+n+1)(\beta+n+1)\right]N_n^{(0)} + z = 0.
\end{equation}
Further finite differencing renders the equation homogeneous, and thus coincides with the recurrence relation for $k\ge1$.
\end{proof}

As both numerator and denominator of Drummond's sequence transformation satisfy the same four-term recurrence relation with different initial conditions, the algorithm is then a race for the maximal solutions of the recurrence relation, potentially present in both numerator and denominator, to dominate over any amount of the minimal solutions of the recurrence relation present in both numerator and denominator, up to a pre-determined absolute or relative accuracy.

Figure~\ref{fig:stability} applies Drummond's sequence transformation to Eq.~\eqref{eq:3F0special} with $\alpha=\beta=1$ and $z=8$ using both the finite difference formul\ae~in Eq.~\eqref{eq:Drummondk^2} and the linear complexity recurrence just derived in Eq.~\eqref{eq:Drummondk}. The numerical results are compared with values that are computed in extended precision, using $256$-bit floating-point arithmetic. It is clear that the linear complexity recurrence relations improve the stability of the algorithm to the point where a simple stopping criterion may be created. In particular, if the computed values of the sequence satisfy
\begin{equation}
|T_n^{(k)} - T_n^{(k-1)}| < 10|T_n^{(k)}|\epsilon_{\rm mach}\quad{\rm and}\quad |T_n^{(k-1)} - T_n^{(k-2)}| < 10|T_n^{(k-1)}|\epsilon_{\rm mach},
\end{equation}
where $\epsilon_{\rm mach} \approx 2.22\times10^{-16}$, then the recurrence is exited, returning the value of $T_n^{(k)}$. Compared with, e.g.~\cite[\S 4]{Gaudreau-Slevinsky-Safouhi-34-B65-12}, this is a rather simple test for convergence that appears to work well in practice.

\begin{figure}[htbp]
\begin{center}
\begin{tabular}{cc}
\hspace*{-0.6cm}\includegraphics[width=0.53\textwidth]{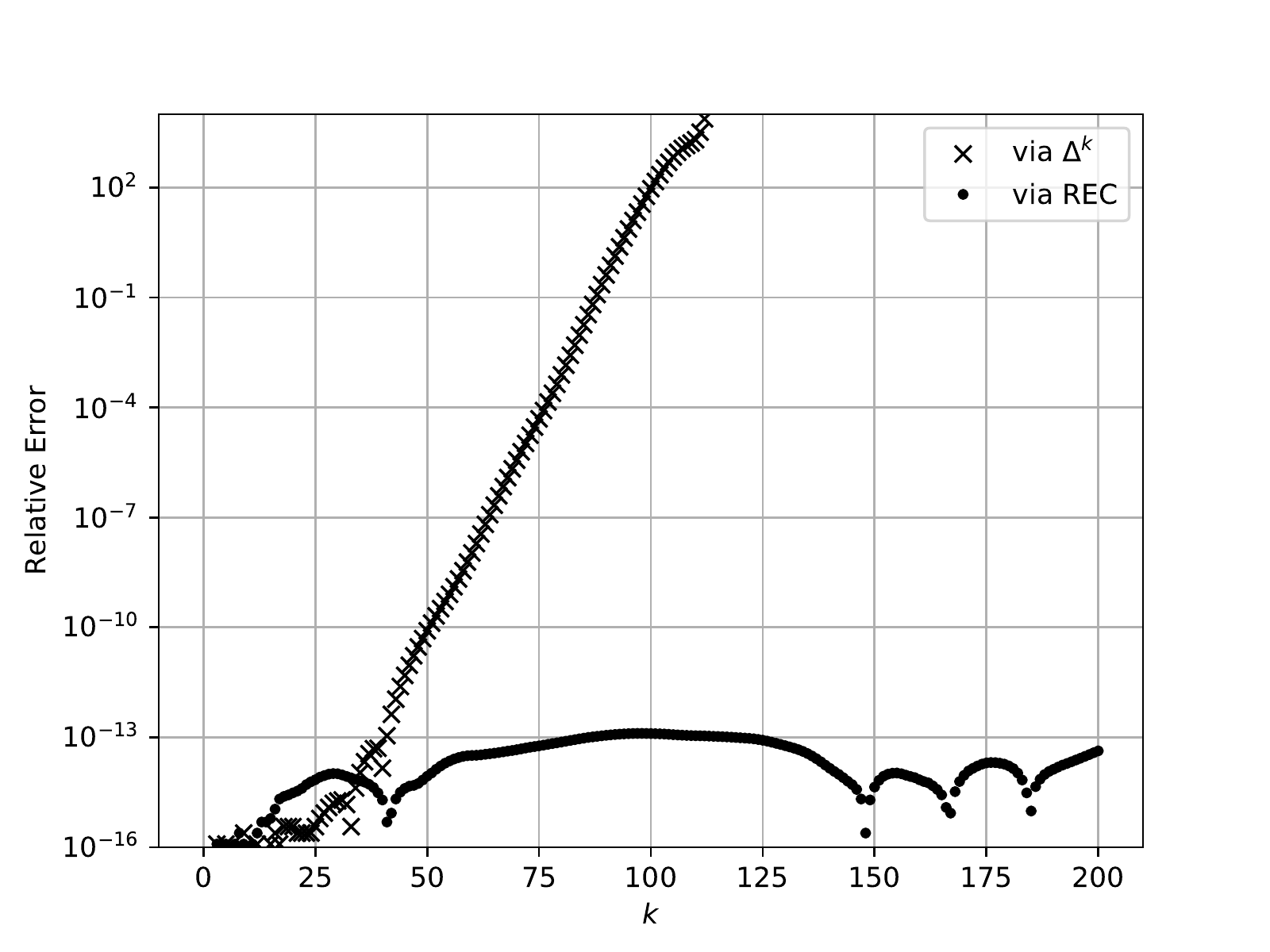}&
\hspace*{-1.1cm}\includegraphics[width=0.53\textwidth]{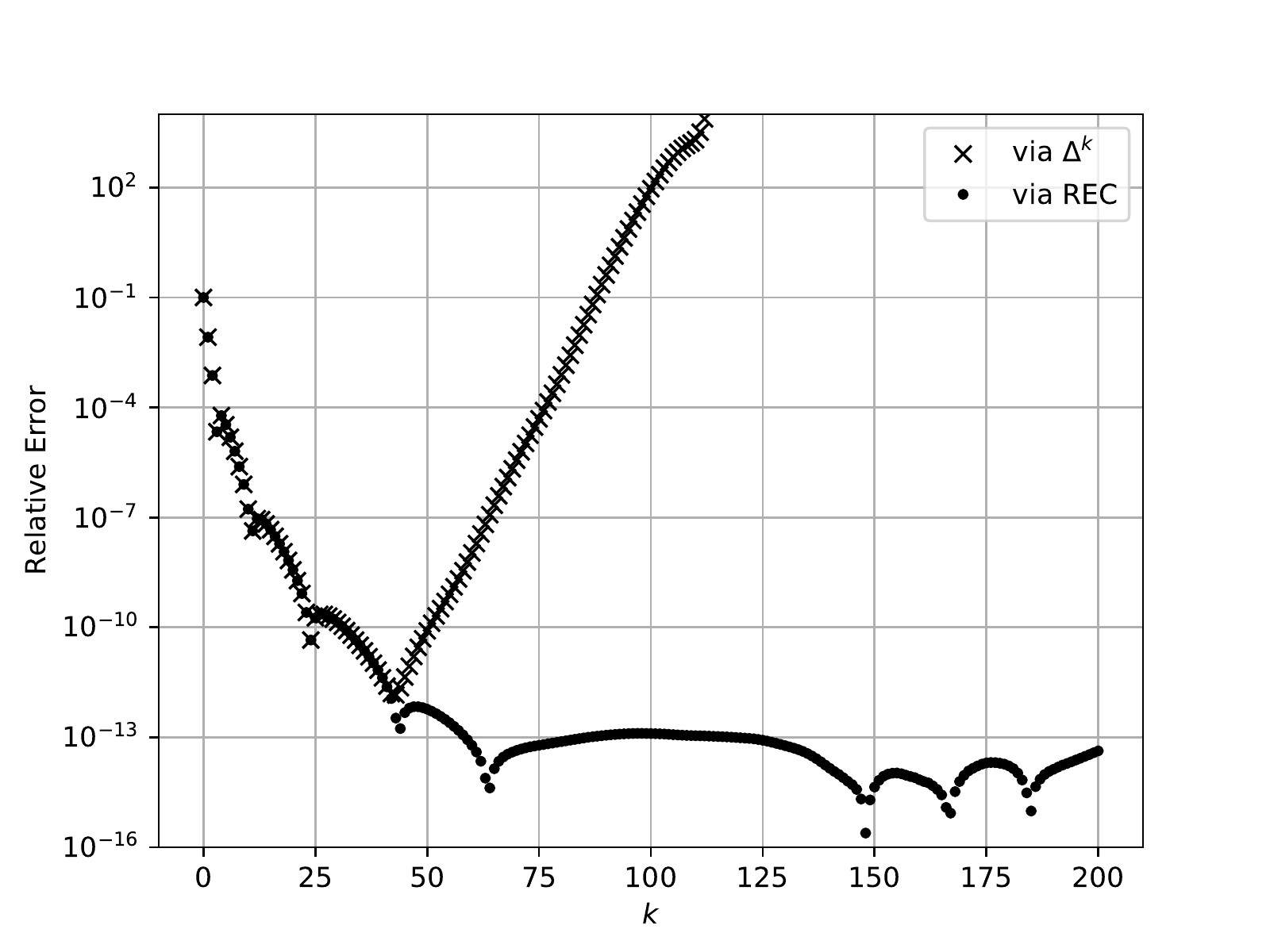}\\
\end{tabular}
\caption{Stability of $T_0^{(k)}$ summing Eq.~\eqref{eq:3F0special} for $\alpha=\beta=1$ and $z=8$, as implemented by Eq.~\eqref{eq:Drummondk^2} and Eq.~\eqref{eq:Drummondk}. Left: the relative error with an extended precision evaluation of $T_0^{(k)}$. Right: the relative error with an extended precision evaluation of a ``true solution,'' $T_0^{(2000)}$, which appears to satisfy $|T_0^{(2000)}-T_0^{(1999)}|/|T_0^{(2000)}| < 1.54\times10^{-49}$.}
\label{fig:stability}
\end{center}
\end{figure}

The recurrence relation in Eq.~\eqref{eq:Drummondk} is equally valid for $\alpha, \beta, z\in\C$. Since the sequence $a_k$ in Eq.~\eqref{eq:3F0special} is terminating if $\alpha$ or $\beta$ is a non-positive integer, we expect the transformation to be exact for such sequences. Indeed, this is the case, and the recurrence relation in Eq.~\eqref{eq:Drummondk} also terminates for such parameter values. There is no numerical difficulty in evaluating the recurrence relations for $z=0$, though the singularity in the series and its analytic continuation via Lommel functions prevents any meaningful approximation from being extracted.

For alternating power series such as Eq.~\eqref{eq:3F0special}, it is readily verified that the denominators in Drummond's sequence transformation are proportional to polynomials with positive coefficients. This ensures that the poles of the approximant are off the positive real axis. When a function has a branch cut, it is well known that rational approximations tend to localize their roots and poles alongside to mimic the cut. This is visible in Figure~\ref{fig:ComplexPhasePortrait}, where the factorially divergent power series in Eq.~\eqref{eq:3F0special} has a branch cut along the negative real axis. For Pad\'e approximants to Stieltjes series, the denominators of diagonal sequences the Pad\'e table may be related to orthogonal polynomials with respect to a positive measure on the branch cut. In such cases, it has been shown that the roots and poles of the Pad\'e approximants are on the cut (and nowhere else in the complex plane). Less is known about the singularity structure of Pad\'e-type approximants.

\begin{figure}[htbp]
\begin{center}
\hspace*{-0.2cm}\includegraphics[width=\textwidth]{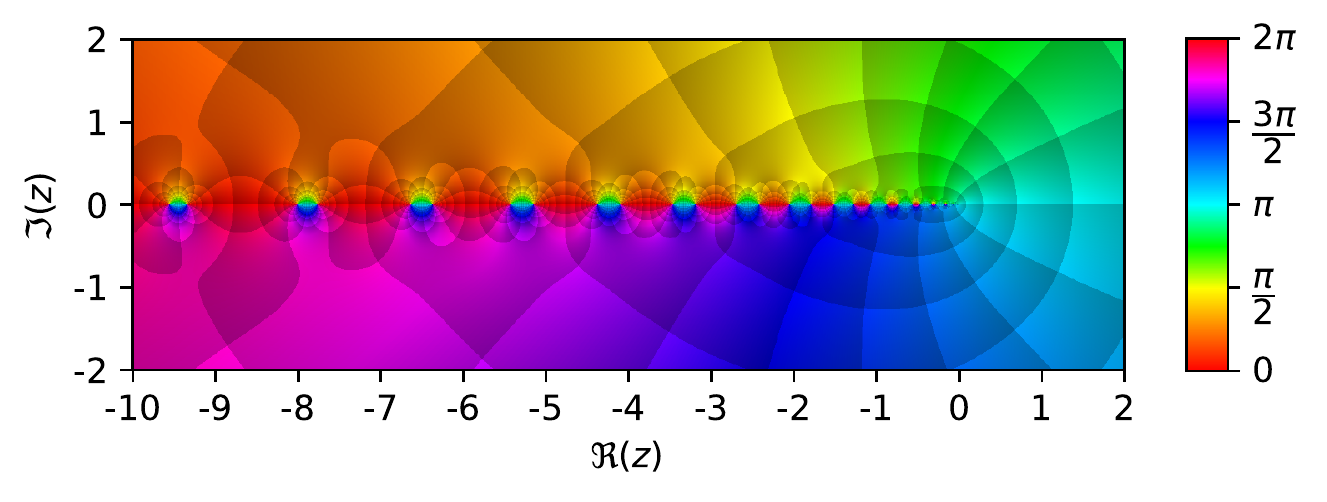}
\caption{Complex phase portrait of $T_0^{(1000)}(z)-1$ approximating the summation of Eq.~\eqref{eq:3F0special} for $\alpha=\beta=1$. This phase portrait illustrates that the roots and poles of the Pad\'e-type approximant appear to interlace and are aligned along the negative real axis, simulating the branch cut of Eq.~\eqref{eq:3F0special}. The complex phase portrait of $T_0^{(1000)}(z)-1$ is shown instead of $T_0^{(1000)}(z)$ to pronounce the contrast as one winds about the origin; the contrast in $T_0^{(1000)}(z)$ itself is less pronounced and the roots are even closer to the poles.}
\label{fig:ComplexPhasePortrait}
\end{center}
\end{figure}

\section{Numerical Discussion}

In Figure~\ref{fig:lambdaerror}, the pointwise relative error in the double precision numerical evaluation of $\lambda_\delta(\kb)$ by both the Maclaurin series and the asymptotic formula is presented for dimensions $d=1,2,3$ and all admissible singularity strengths $\alpha\in[0,d+2)$. The numerical results are compared with the Maclaurin series evaluated in extended precision, using $256$-bit floating-point arithmetic. It appears that the heuristic $k\delta=6$ separates the accurate and inaccurate regions of each algorithm, uniformly for every $d$ and all $\alpha$. This makes a unified algorithm directly available: if $\abs{k\delta}<6$, then we use the Maclaurin series; otherwise, we use the asymptotic formula. On a MacBook Pro (Mid 2014) with a 2.8 GHz Intel Core i7-4980HQ processor, the computation of $\lambda_\delta(\kb)$ with $k\delta = 6$, $\alpha=2$, and $d=3$ takes approximately $0.12\,\mu{\rm s}$ via the Maclaurin series and $4.0\,\mu{\rm s}$ via the asymptotic formula, averaging results over $10^6$ evaluations. We choose the value $k\delta = 6$ because this value roughly requires the most flops.

\begin{figure}[htbp]
\begin{center}
\includegraphics[width=\textwidth]{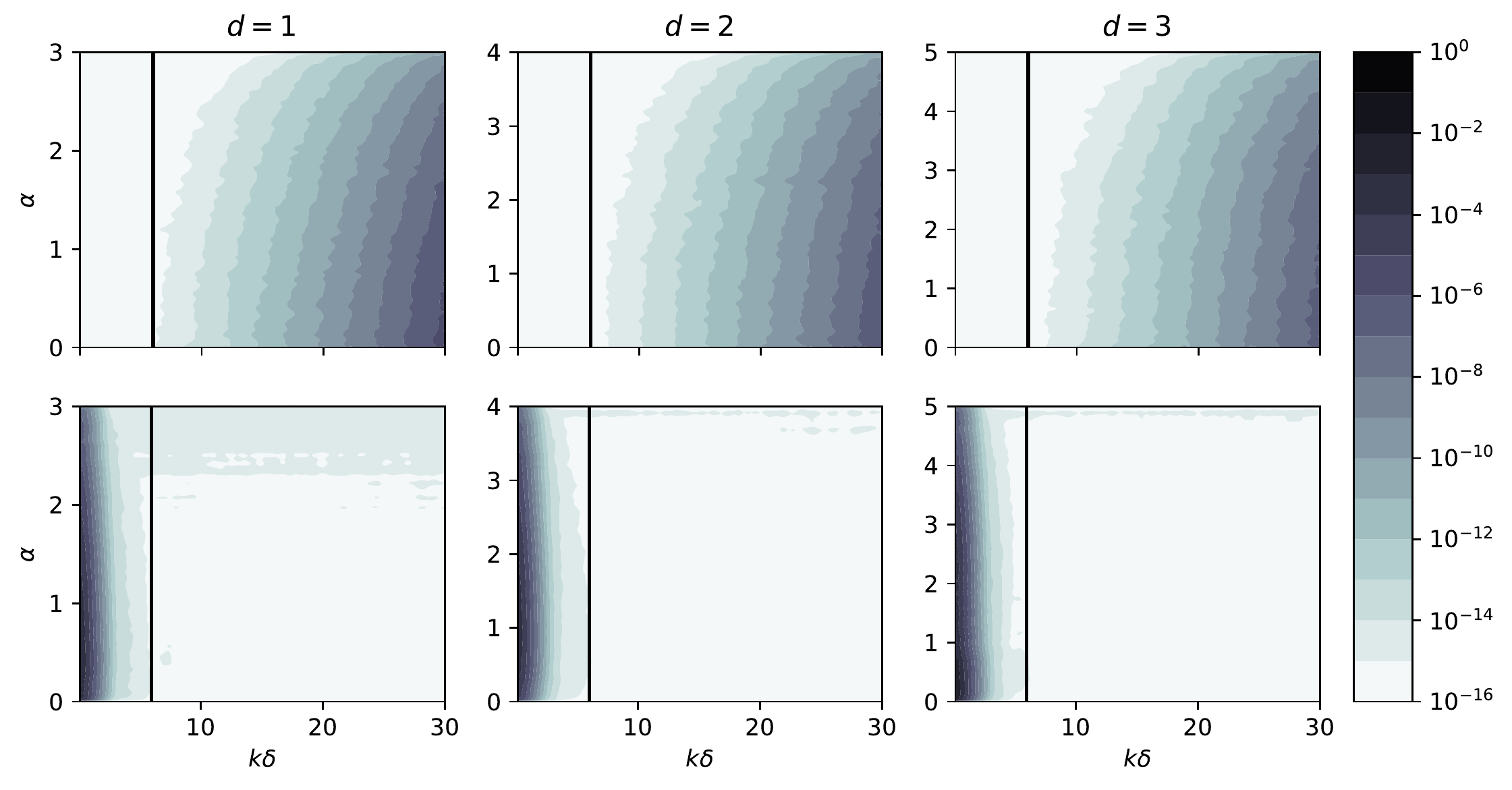}
\caption{Pointwise relative error compared with an extended precision evaluation of $\lambda_\delta(\kb)$ of Top: the Maclaurin series. Bottom: the asymptotic series resummed by Drummond's transformation. In all plots, the black vertical line $k\delta = 6$ is a heuristic division between the accurate and inaccurate regions of each algorithm.}
\label{fig:lambdaerror}
\end{center}
\end{figure}

Numerically, the main advantage of the algorithm in this paper over the hybrid algorithm of~\cite{Du-Yang-332-118-17} is that the evaluation of the Fourier spectra are now trivially parallelizable, unleashing the full potential of modern supercomputers. Of secondary importance, the relative error of each eigenvalue is independently controllable; this cannot be done using a time-stepping scheme for the linear ordinary differential equation that is satisfied by the generalized hypergeometric function in Eq.~\eqref{eq:FourierSpectraMAC}. Furthermore, our codes are freely available in the {\sc Julia} package {\tt SingularIntegralEquations.jl}, which implements nonlocal diffusion as an infinite-dimensional banded (diagonal) operator with the deferred evaluation of the entries occurring on a need-to-be-evaluated basis. Combined with time-stepping schemes for equations of evolution, this framework allows for the simulation of many nonlocal phenomena, some of which are explored in~\cite{Du-Yang-332-118-17}.

Many extensions of this work are natural. In multiple dimensions, nonlocal operators of vector calculus developed by~\cite{Du-et-al-23-493-13} define similar eigenvalue problems when acting on the Fourier basis. Applications of nonlocal calculus may involve kernels that do not have a weak algebraic singularity, such as a logarithmic singularity or a Gaussian envelope. This is reasonable since the family of permissible kernels may be determined by admissibility in a function space rather than having a particular analytic structure. However, weakly singular algebraic kernels will always be important due to their correspondence with bounded operators in between Sobolev spaces. Furthermore, it is also of interest to examine the effect of applying nonlocal operators to different bases; orthogonal polynomial bases such as Chebyshev and Legendre (Jacobi polynomials in general) are ubiquitous in numerical analysis and are anticipated to be useful to solve univariate nonlocal problems without periodicity. It is not inconceivable that robust numerical algorithms may be developed in these cases as well based on Maclaurin and resummed asymptotic series.

Theorem~\ref{theorem:DrummondRec} offers the possibility of sequence transformations such as Drummond's or others of the same structure (\cite{Drummond-6-69-72,Levin-B3-371-73,Sidi-7-37-81,Weniger-10-189-89}) to be stabilized by linear complexity recurrence relations. What are the general classes of sequences for which such relations are obtainable? Bespoke tailoring of sequence transformations with the same structure has been used by~\cite{Sidi-7-37-81} to create Pad\'e approximants to some generalized hypergeometric functions; can this methodology be generalized to generate Pad\'e approximants to any $\pFq{p}{q}$? Diagonal sequences in the Pad\'e table approximating Stieltjes series satisfy the three-term recurrence relation corresponding to the orthogonal polynomials and their associated polynomials with respect to the Stieltjes measure. However, other than a few exceptional circumstances such as for classical orthogonal polynomials, the computation of recurrence {\em coefficients} cannot be taken for granted. On the other hand,~\cite{Borghi-Weniger-94-149-15} prove that Pad\'e-type approximants with a judicious choice of poles have favourable convergence properties over Pad\'e approximants, notwithstanding the higher order conditions satisfied by the latter. Due to the prominence of $\pFq{p}{q}$ in the theory of special functions, we believe it is worthwhile to explore stable and reliable linear complexity algorithms for their Pad\'e and Pad\'e-type approximants.

\section*{Acknowledgments}

We thank Hadrien Montanelli and Qiang Du for discussions on nonlocal diffusion. We thank the Natural Sciences and Engineering Research Council of Canada (RGPIN-2017-05514) for supporting this research.

\bibliography{/Users/Mikael/Bibliography/Mik}

\begin{thebibliography}{}

\bibitem[Borghi \& Weniger(2015)Borghi \& Weniger]{Borghi-Weniger-94-149-15}
{\sc Borghi, R. \& Weniger, E.~J.} (2015)
\newblock Convergence analysis of the summation of the factorially divergent
  {E}uler series by {P}ad\'e approximants and the delta transformation.
\newblock {\em Appl. Num. Math.}, {\bf 94}, 149--178.

\bibitem[Brezinski(1980)Brezinski]{Brezinski-80}
{\sc Brezinski, C.} (1980)
\newblock {\em Pad\'e-Type Approximation and General Orthogonal Polynomials\/}.
\newblock Birkh\"auser.

\bibitem[Drummond(1972)Drummond]{Drummond-6-69-72}
{\sc Drummond, J.~E.} (1972)
\newblock A formula for accelerating the convergence of a general series.
\newblock {\em Bull. Austral. Math. Soc.}, {\bf 6}, 69--74.

\bibitem[Du {\em et~al.}(2012)Du, Gunzburger, Lehoucq, \&
  Zhou]{Du-et-al-54-667-12}
{\sc Du, Q., Gunzburger, M., Lehoucq, R.~B. \& Zhou, K.} (2012)
\newblock Analysis and approximation of nonlocal diffusion problems with volume
  constraints.
\newblock {\em SIAM Rev.}, {\bf 54}, 667--696.

\bibitem[Du {\em et~al.}(2013)Du, Gunzburger, Lehoucq, \&
  Zhou]{Du-et-al-23-493-13}
{\sc Du, Q., Gunzburger, M., Lehoucq, R.~B. \& Zhou, K.} (2013)
\newblock A nonlocal vector calculus, nonlocal volume-constrained problems, and
  nonlocal balance laws.
\newblock {\em Math. Mod. Meth. Appl. Sci.}, {\bf 23}, 493--540.

\bibitem[Du \& Tian(2018)Du \& Tian]{Du-Tian-1805-08261}
{\sc Du, Q. \& Tian, X.} (2018)
\newblock Mathematics of smoothed particle hydrodynamics, {P}art {I}: a
  nonlocal {S}tokes equation.
\newblock arXiv:1805.08261.

\bibitem[Du \& Yang(2017)Du \& Yang]{Du-Yang-332-118-17}
{\sc Du, Q. \& Yang, J.} (2017)
\newblock Fast and accurate implementation of {F}ourier spectral approximations
  of nonlocal diffusion operators and its applications.
\newblock {\em J. Comp. Phys.}, {\bf 332}, 118--134.

\bibitem[Estrada-Rodriguez {\em et~al.}(2018)Estrada-Rodriguez, Gimperlein, \&
  Painter]{Estrada-et-al-78-1155-18}
{\sc Estrada-Rodriguez, G., Gimperlein, H. \& Painter, K.~J.} (2018)
\newblock Fractional {P}atlak--{K}eller--{S}egel equations for chemotactic
  superdiffusion.
\newblock {\em SIAM J. Appl. Math.}, {\bf 78}, 1155--1173.

\bibitem[Gaudreau {\em et~al.}(2012)Gaudreau, Slevinsky, \&
  Safouhi]{Gaudreau-Slevinsky-Safouhi-34-B65-12}
{\sc Gaudreau, P., Slevinsky, R.~M. \& Safouhi, H.} (2012)
\newblock Computation of tail probability distributions via extrapolation
  methods and connection with rational and {P}ad\'e approximants.
\newblock {\em SIAM J. Sci. Comput.}, {\bf 34}, B65--B85.

\bibitem[Gilboa \& Osher(2009)Gilboa \& Osher]{Gilboa-Osher-7-1005-09}
{\sc Gilboa, G. \& Osher, S.} (2009)
\newblock Nonlocal operators with applications to image processing.
\newblock {\em Multiscale Model. Simul.}, {\bf 7}, 1005--1028.

\bibitem[Gradshteyn \& Ryzhik(2007)Gradshteyn \& Ryzhik]{Gradshteyn-Ryzhik-07}
{\sc Gradshteyn, I.~S. \& Ryzhik, I.~M.} (2007)
\newblock {\em Table of Integrals, Series, and Products, Seventh Edition\/}.
\newblock Burlington, MA: Elsevier Academic Press.

\bibitem[Lanczos(1964)Lanczos]{Lanczos-1-86-64}
{\sc Lanczos, C.} (1964)
\newblock A precision approximation of the gamma function.
\newblock {\em J. SIAM Series B Numer. Anal.}, {\bf 1}, 86--96.

\bibitem[Levin(1973)Levin]{Levin-B3-371-73}
{\sc Levin, D.} (1973)
\newblock Development of non-linear transformations for improving convergence
  of sequences.
\newblock {\em Int. J. Comput. Math.}, {\bf B3}, 371--388.

\bibitem[Levin \& Sidi(1981)Levin \& Sidi]{Levin-Sidi-9-175-81}
{\sc Levin, D. \& Sidi, A.} (1981)
\newblock Two new classes of non-linear transformations for accelerating the
  convergence of infinite integrals and series.
\newblock {\em Appl. Math. Comput.}, {\bf 9}, 175--215.

\bibitem[Michel \& Stoitsov(2008)Michel \&
  Stoitsov]{Michel-Stoitsov-178-535-08}
{\sc Michel, N. \& Stoitsov, M.~V.} (2008)
\newblock Fast computation of the {G}auss hypergeometric function with all its
  parameters complex with application to the {P\"o}schl--{T}eller--{G}inocchio
  potential wave functions.
\newblock {\em Comp. Phys. Commun.}, {\bf 178}, 535--551.

\bibitem[Olver {\em et~al.}(2010)Olver, Lozier, Boisvert, \&
  Clark]{Olver-et-al-NIST-10}
{\sc Olver, F. W.~J., Lozier, D.~W., Boisvert, R.~F. \& Clark, C.~W.} (eds)
  (2010)
\newblock {\em NIST Handbook of Mathematical Functions\/}.
\newblock Cambridge, UK: Cambridge U. P.

\bibitem[Sidi(1981)Sidi]{Sidi-7-37-81}
{\sc Sidi, A.} (1981)
\newblock A new method for deriving {P}ad\'e approximants for some
  hypergeometric functions.
\newblock {\em J. Comp. Appl. Math.}, {\bf 7}, 37--40.

\bibitem[Silling(2000)Silling]{Silling-48-175-00}
{\sc Silling, S.~A.} (2000)
\newblock Reformulation of elasticity theory for discontinuities and long-range
  forces.
\newblock {\em J. Mech. Phys. Solids\/}, {\bf 48}, 175--209.

\bibitem[Slevinsky {\em et~al.}(2018)Slevinsky, Montanelli, \&
  Du]{Slevinsky-Montanelli-Du-372-893-18}
{\sc Slevinsky, R.~M., Montanelli, H. \& Du, Q.} (2018)
\newblock A spectral method for nonlocal diffusion operators on the sphere.
\newblock {\em J. Comp. Phys.}, {\bf 372}, 893--911.

\bibitem[Slevinsky \& Safouhi(2012)Slevinsky \&
  Safouhi]{Slevinsky-Safouhi-60-315-12}
{\sc Slevinsky, R.~M. \& Safouhi, H.} (2012)
\newblock A comparative study of numerical steepest descent, extrapolation, and
  sequence transformation methods in computing semi-infinite integrals.
\newblock {\em Numer. Algor.}, {\bf 60}, 315--337.

\bibitem[Watson(1966)Watson]{Watson-66}
{\sc Watson, G.~N.} (1966)
\newblock {\em A Treatise on the Theory of {B}essel Functions\/}.
\newblock Cambridge, England: Cambridge University Press, Second Edition.

\bibitem[Weniger(1989)Weniger]{Weniger-10-189-89}
{\sc Weniger, E.~J.} (1989)
\newblock Nonlinear sequence transformations for the acceleration of
  convergence and the summation of divergent series.
\newblock {\em Comput. Phys. Rep.}, {\bf 10}, 189--371.

\bibitem[Zheng {\em et~al.}(2017)Zheng, Hu, Du, \&
  Zhang]{Zheng-et-al-39-A1951-17}
{\sc Zheng, C., Hu, J., Du, Q. \& Zhang, J.} (2017)
\newblock Numerical solution of the nonlocal diffusion equation on the real
  line.
\newblock {\em SIAM J. Sci. Comput.}, {\bf 39}, A1951--A1968.

\end{thebibliography}

\end{document}